\documentclass{article}

\usepackage[utf8]{inputenc}
\usepackage{cite}
\usepackage{amsmath,amssymb,amsfonts}
\usepackage{algorithmic}
\usepackage{graphicx}
\usepackage{textcomp}
\usepackage{xcolor}
\usepackage{lipsum}
\usepackage{latexsym,amsfonts,amssymb,amsthm,amsmath,mathtools}
\usepackage{bm}
\usepackage{mathrsfs}
\usepackage{dsfont}
\usepackage[T1]{fontenc}

\DeclareMathOperator{\TV}{TV} 

\newcommand{\DE}{\bm{\Delta}}

\newcommand{\p}{\mathbb{P}} 
\newcommand{\Ex}{\mathbb{E}} 

\newtheorem{proposition}{Proposition}
\newtheorem{definition}{Definition}
\newtheorem{theorem}{Theorem}
\newtheorem{lemma}{Lemma}

\newtheorem{corollary}{Corollary}

\definecolor{AfonsoBlue}{RGB}{30,65,123}

\title{Guarantees for Spontaneous Synchronization on Random Geometric Graphs. }
\author{Pedro Abdalla, Afonso S. Bandeira and Clara Invernizzi \\Department of Mathematics, ETH Z\"urich}

\begin{document}

\maketitle

\begin{abstract}
The Kuramoto model is a classical mathematical model in the field of non-linear dynamical systems that describes the evolution of coupled oscillators in a network that may reach a synchronous state. The relationship between the network's topology and whether the oscillators synchronize is a central question in the field of synchronization, and random graphs are often employed as a proxy for complex networks. On the other hand, the random graphs on which the Kuramoto model is rigorously analyzed in the literature are homogeneous models and fail to capture the underlying geometric structure that appears in several examples.

In this work, we leverage tools from random matrix theory, random graphs, and mathematical statistics to prove that the Kuramoto model on a random geometric graph on the sphere synchronizes with probability tending to one as the number of nodes tends to infinity. To the best of our knowledge, this is the first rigorous result for the Kuramoto model on random geometric graphs.

\end{abstract}

\section{Introduction}
Spontaneous synchronization is an ubiquitous phenomenon in nature.  The experimental investigation of this phenomenon dates back to the work of Huygens in 1665, who observed that two pendulums held by the same wooden beam tend to swing in anti-phase after a certain amount of time, regardless of their initial position. The study of the synchronization phenomenon in systems with a large number of coupled oscillators started many years later, following the seminal work of Wiener \cite{wiener1948cybernetics}. He was interested in the rhythms of the brain and believed that other biological mechanisms could be modeled similarly. An example of this phenomenon is when thousands of fireflies start to flash on and off in unison \cite{strogatz1993coupled}. While Wiener's ideas were remarkable, his proposed mathematical models were too complex to analyze. In a seminal work in 1975 \cite{kuramoto1975self}, Kuramoto introduced the first mathematically tractable model to study synchronization in a system of all-to-all coupled phase oscillators: Consider a system of $n$ all-to-all connected phase oscillators represented by the time-varying angles $\theta_1(t),\ldots,\theta_n(t)$, with natural frequencies $\omega_1,\ldots,\omega_n$. The behaviour of the oscillators is described by the following dynamical system 
\begin{equation*}
\frac{d}{dt}\theta_i(t) = \omega_i -\frac{K}{n}\sum_{j=1}^n\sin(\theta_i-\theta_j),
\end{equation*}
where $K>0$ is the coupling strength of the system. The system of oscillators synchronizes when, as time $t$ approaches infinity, each angle $\theta_1(t), \ldots, \theta_n(t)$ converges to the same limit $\theta$ for almost all initializations $\theta_1(0), \ldots, \theta_n(0)$. In the original work, Kuramoto analyzed the dynamical system using non-rigorous techniques from statistical physics and showed that when the natural frequencies are drawn at random, the oscillators achieve synchronization with probability approaching one as $n$ approaches infinity, provided that $K$ is sufficiently large.

Despite the significant progress in mathematically describing the behaviour of coupled oscillators, the original Kuramoto model has a limitation: It assumes that all oscillators interact equally among themselves. However, many complex systems, such as flashing fireflies are better described by complex networks with heterogeneous structures.
Therefore, a generalization of the Kuramoto model that incorporates the topology of the oscillator system is necessary.

To simplify the analysis, in this work, we assume that all oscillators have the same natural frequency, specifically $w_i=w$ for all $i\in [n]$. In this case, by moving into a rotation frame $\theta(t)\rightarrow \theta(t)-wt$ and normalizing the coupling strength to $K=1$, the dynamical system becomes
\begin{equation*}
\frac{d \theta_i(t)}{dt} = \frac{1}{n}\sum_{j=1}^n\sin(\theta_i-\theta_j).
\end{equation*}
Next, consider a graph $G=(V,E)$ with adjacency matrix $A$ with entries $a_{ij} \in \{0,1\}$ that represent the connections in the system. The behaviour of the coupled oscillators associated to the graph $G$ is modeled by the following generalized dynamical system
\begin{equation*}
\frac{d\theta_i(t)}{dt} =  \frac{1}{n}\sum_{i,j=1}^n a_{ij}\sin(\theta_i-\theta_j).
\end{equation*}
Notice that the original Kuramoto model is the particular case when the graph $G$ is the complete graph. 

In a remarkable work, D\"orfle, Chertkov, and Bullo \cite{dorfler2013synchronization,dorfler2014synchronization} considerably simplified the problem by introducing an alternative viewpoint: The coupled oscillators system can be interpreted as a system of particles that attempts to minimize the following energy function
\begin{equation}
\label{eq:energy_function}
    E_G(\theta):= \sum_{i,j}a_{ij}(1-\cos(\theta_i-\theta_j)),
\end{equation}
and the standard Cauchy problem associated with this energy function can be written as
\begin{equation*}
\frac{d\theta_i(t)}{dt} = - \frac{\partial E(\theta_1,\ldots,\theta_n)}{\partial \theta_i}, \quad \theta_i=\theta_i(0).
\end{equation*}
From the equation above, the intuitive meaning of the Kuramoto model is that the evolution moves according to the gradient flow of the energy function, i.e, the direction of maximum decrease of $E_G(\theta)$.
It is easy to see that the global minimum of this energy function is $\theta_1=...=\theta_n = 0$ modulo a global translation, corresponding to the synchronous state.
Mathematically, determining whether the coupled oscillators will synchronize boils down to show whether the energy function $E_{G}(\theta)$ admits a local minimum that is not the global minimum, also known as spurious minimum.

We note that the applications of the Kuramoto model in complex networks are extensive, spanning many branches of science, including power grids, wireless networks, neuroscience, social networks \cite{acebron2005kuramoto,dorfler2013synchronization,rodrigues2016kuramoto,dorfler2014synchronization}. The mathematical problem of understanding how the topology of the graph influences the absence of spurious local minima is an active topic of research in applied dynamical systems \cite{kassabov2022global,abdalla2022expander,taylor2012,ling2019landscape,kassabov2021}. Additionally, characterizing  the optimization landscape of non-convex functions is an active topic as well. Indeed, there is a significant body of work dedicated to studying of the optimization landscape of non-convex functions that often arise in machine learning and data science (e.g.~\cite{ge2016matrix,ge2017optimization}). Some works have directly connected these problems to the study of energy functions similar to the one addressed here~\cite{boumal2016nonconvex,ling2019landscape}.

\subsection{Kuramoto Model on random graphs}
Some very dense graphs do not synchronize. In fact, a line of work \cite{taylor2012,ling2019landscape,kassabov2021} has investigated how large the minimum degree of a deterministic graph needs to be in order to guarantee synchronization regardless of its topology. A trivial necessary condition is that the underlying graph is connected, however it is far from sufficient. A notable result from Kassabov, Strogatz and Townsend \cite{kassabov2021} states that if the graph $G$ has minimum degree at least $0.75(n-1)$, then it must synchronize regardless of the topology of the graph. On the other hand, explicit examples demonstrate that $0.75(n-1)$ cannot be improved below $0.68(n-1)$ \cite{ling2019landscape,taylor2012}. In other words, the graph needs to be quite dense to ensure synchronization. Typically, as it turns out, networks are much sparser, and the topology of the graph $G$ must be taken into account.

On a probabilistic side, it has been observed that many sparse networks synchronize. In particular, consider the well-known Erd\H{o}s-R\'enyi random graph model $G(n,p)$, \cite{erdhos} in which each edge is placed independently with probability $p$. Ling, Xu and the second author \cite{ling2019landscape} showed that if $p = c(\log n) n^{-1/3}$, for a suitable constant $c>32$, a graph sampled from $G(n,p)$ synchronizes with probability approaching to one as the $n$ goes to infinity, i.e, with high probability. They also conjectured that if $p$ is right above the connectivity threshold, $p = (\log n)n^{-1}$, then $G(n,p)$ should synchronize with high probability. This conjecture was confirmed by Kassabov, Souza, Strogatz, Townsend and the first two authors \cite{abdalla2022expander}. The main result here is that connectivity is the bottleneck to ensure synchronization, providing evidence that connected graphs that do not synchronize are somewhat rare. This property does not appear to be exclusively to Erd\H{o}s-R\'enyi graphs; random $d$-regular graphs are globally synchronizing as long as $d\ge 600$. It is worth noting that for $d\ge 3$, a random $d$-regular graph is connected with high probability, but it remains an open question whether it is also globally synchronizing \cite{abdalla2022expander}.

While the Erd\H{o}s-R\'enyi random graphs capture the sparsity of complex networks, they fail to capture heterogeneity and geometry. As argued by Rodrigues, Peron, Ji and Kurths \cite{rodrigues2016kuramoto}, in a social network, two friends $A$ and $B$, are more likely to connect if they are already connected to a common friend, $C$. Similarly, a flashing firefly is more likely to influence nearby fireflies than the ones far away. Motivated by this observation, in this work, we analyze the presence of spurious local minima of the energy function when the graph $G$ is randomly generated with an underlying geometric structure, the so-called random geometric graphs on the sphere. These graphs consist of $n$ vertices sampled by generating $n$ points uniformly at random on the unit sphere, and connecting vertices with an edge if their Euclidean distance is smaller than a specified threshold. Random geometric graphs on the sphere capture the heterogeneous property that two nodes are more likely to connect if both are connected with a third node.

Our main contribution is to provide a rigorous treatment for the random geometric graphs by exploiting their connections with random matrix theory and recent work on hypothesis testing on random graph models \cite{TestHighDim21}. Our main result establishes that a random geometric graph with $n$ nodes drawn uniformly at random from the unit sphere $\mathbb{S}^{d-1}$ synchronizes, with high probability, as long as $d$ is sufficiently large as a function of $n$. We state here a simplified statement, a more detailed version and its proof is presented in Sections \ref{Sec:RGG} and \ref{sec:sublinear}. In what follows, $\Tilde{\Omega}(.)$ hides poly-logarithmic factors.

\begin{theorem}(Simplified Statement of the Main Result Theorem \label{MainThm})
Let $G$ be a random geometric graph with $n$ nodes on the unit sphere $\mathbb{S}^{d-1}$ such that each edge is connected with marginal probability $p$ (see the rigorous Definition \ref{def:RGG_sphere}). If $p$,$n$ and $d$ satisfy at least one of the two regimes
\begin{enumerate}
\item $np=\Omega(\log^2(n))$ and $d=\Tilde{\Omega}(n^2p^2)$.
\item $np = \Omega(\log^{10}(n))$ and $d=\Omega(\log^3 (n))$.
\end{enumerate}
Then the graph $G$ synchronizes with probability tending to one as $n$ goes to infinity.
\end{theorem}

We conclude this section with a discussion of our main result. It shows that if $p=\Tilde{\Theta}\left(\frac{1}{n}\right)$ and $d = \Tilde{\Omega}\left(n^2p^2\right) =\Tilde{\Omega}\left(1\right)$, then the random geometric graph on the sphere $\mathbb{S}^{d-1}$ with marginal connectivity probability $p$ synchronizes with high probability. On the other hand, it is natural to expect that if $p$ grows and the other parameters are fixed, then $G$ becomes more likely to synchronize. If we could formally prove that this monotonicity holds, then our result would imply that for any fixed $p= \Tilde{\Omega}\left(\frac{1}{n}\right)$ and $d = \Tilde{\Omega}\left(1\right)$, the graph synchronizes. This is essentially the best we can hope for (up to logarithmic factors) because below this threshold, the graph will not be connected with non-vanishing probability. An important outcome of our result is that it suggests that connectivity may indeed be the bottleneck for synchronization of random geometric graphs, similar to the Erd\H{o}s-R\'enyi model. However, it is important to note that our current techniques cannot establish global synchrony when the dimension $d$ is small. Analyzing random geometric graphs on the sphere becomes significantly more challenging in this case as they differ significantly from the Erd\H{o}s-R\'enyi model, and we leave it as an open question.

The challenge arises because the proof technique relies on comparisons between the geometric and Erd\H{o}s-R\'enyi random graph models, which impose upper bounds on $p$. While one expects that, on average, larger $p$ aids synchronization (even if it might make a geometric graph more easily distinguishable from an Erd\H{o}s-R\'enyi one), we were not able to establish such a monotonicity result, and we leave it as an open problem. An important difficulty is that adding edges may destroy the synchronization property. In fact, as pointed out in \cite[Remark 1.7]{abdalla2022expander}, it is easy to see that a tree is globally synchronizing, while a cycle of length five is not.

\subsection*{Notation}
The standard unit sphere in $\mathbb{R}^d$ is denoted by $\mathbb{S}^{d-1}$ as usual. The matrix $J_n$ represents the $n\times n$ all ones matrix and $I_n$ denotes the $n\times n$ identity matrix. For vectors $u,v$ of same size, the inner product $\langle u,v\rangle$ corresponds to the standard inner product, and $\|v\|_2$ denotes the standard Euclidean norm. For a matrix A, $\|A\|$ denotes the standard operator norm. We consider the standard "Big-O" notation (Landau notation) and we write $\Tilde{\Omega}(.)$ to hide log factors. An event on a random graph $G$ with $n$ vertices occurs with high probability if the probability of occurrence  converges to one as $n$ goes to infinity.

\section{Preliminaries on the Kuramoto Model}

We will establish spontaneous synchronization of the graphs drawn from our random graph model by showing that they satisfy, with high probability, explicit properties that imply synchronization. One approach, originally developed to analyse the landscape of certain optimization problems arising from the Burer-Monteiro relaxation to semidefinite programming (see~\cite{burer2003nonlinear,burer2005local}), provides a weaker estimate than the one presented in this paper, and is available in an earlier version of our work~\cite{Clara_thesis}.

A second approach, relying on a recent remarkable result by Kassabov, Strogatz and Townsend \cite{kassabov2022global}, provides sharper guarantees. To state their result, we start with the following definition 
\begin{definition}
A graph $G=(V,E)$ with adjacency matrix $A$ is said to be globally synchronizing if all local minima of the associated energy function \eqref{eq:energy_function} correspond to a constant function, $\theta_i =\theta$ for every $i\in [n]$.
\end{definition}
The following theorem establishes a sufficient condition for the graph to synchronize based on the spectral properties of the adjacency matrix and Laplacian matrix of the graph. 

\begin{theorem}\cite[Theorem 6]{kassabov2022global}
\label{thm:kassabovetal}
 Let $G$ be a graph with $n$ vertices and fix $p_o$ such that $0<p_o<1$. Let $A$ and $L$ be the associated adjacency, Laplacian matrix respectively, and define $\DE_A = A - p_oJ_n$, $\DE_L = L + p_oJ_n - np_oI_n$. If the three conditions below are satisfied
    \begin{enumerate}
        \item $\frac{\| \DE_A \|}{np_o} < \frac{1}{12} $,
        \item $\frac{\| \DE_L \|}{np_o} < \frac{1}{4}$,
        \item $\frac{\pi/4}{\sin^{-1} \left(\frac{12\|\DE_A\|}{np_o} \right)} > \frac{\log(n/6)}{\log(\frac{np_o}{2\|\DE_L\|}-1)}+1$,
    \end{enumerate}
then the homogeneous Kuramoto model on the graph $G$ globally synchronizes.
\end{theorem}
In the next section, we prove that these conditions are satisfied for random geometric graphs on the sphere in the regime stated in Theorem \ref{MainThm}.

\section{The Random Geometric Graph on the Sphere}
\label{Sec:RGG}
We start by formally defining the random geometric graph on the sphere. We also include the definition of the Erd\H{o}s-R\'enyi random graph for the sake of completeness.

\begin{definition}[Erd\H{o}s-R\'enyi Random Graph \cite{erdhos}]
Let $n\in \mathbb{N}$ and consider $p$ a fixed probability. The random graph $G(n,p)$ is drawn as follows: For each pair  $\{i,j\} \in [n]\times [n]$ with $i\neq j$, an edge is placed with probability $p$ independently.
\end{definition}
\begin{definition}[Random Geometric Graph on the Sphere \cite{lugosi2017lectures,penrose2003random}]
\label{def:RGG_sphere}
Let $n\in \mathbb{N}$ and $X_1,...,X_n \in \mathbb{S}^{d-1}$ be points drawn independently and uniformly at random on the unit sphere $\mathbb{S}^{d-1}$. The random geometric graph $G(n,p,d)$ is defined as the graph on the vertex set $[n]$ such that the pair $(i,j)$ ($i\neq j$) is an edge of the graph if and only if $\langle X_i,X_j\rangle \geq t_{p,d}$, where $t_{p,d}\in [-1,1]$ is related to $p$ by 
    \begin{equation*}
        p = \p\left(\langle X_i, X_j\rangle \geq t_{p,d}\right).
    \end{equation*}
    Equivalently, $\{i,j\}$ is an edge of the graph if and only if the Euclidean distance between $X_i$ and $X_j$ satisfies $\| X_i- X_j\|_2 \leq \sqrt{2-2t_{p,d}}$.
\end{definition}

A recent fascinating line of work ~\cite{devroye2011high,bubeck2016testing,brennan2020phase,TestHighDim21}  is dedicated to the following question: If we receive one sample from a random graph, either sampled from $G(n,p,d)$ or $G(n,p)$, can we determine with high probability from which graph it was sampled? In other words, is it possible do hypothesis testing with the probability of error going to zero as $n$ goes to infinity?

Clearly, if we cannot distinguish between the $G(n,p)$ and $G(n,p,d)$, then the random graph $G(n,p,d)$ must synchronize with high probability since we already know that the $G(n,p)$ does. Recall that the definition of total variation distance between two random graphs $G$ and $G'$ with set of vertices $V$ is defined by 
\begin{equation*}
    \TV(G,G') = \max_{\mathcal{G}} |\mathbb{P}(G \in \mathcal{G}) - \mathbb{P}(G' \in \mathcal{G}) |,
\end{equation*}
where $\mathcal{G}$ is any set of graphs over $V$. In our context, an important insight from \cite[Theorem 1.2]{TestHighDim21} states that for a fixed constant $\alpha>0$, if $\frac{\alpha}{n}< p < \frac{1}{2}$ and $d = \Tilde{\Omega}(n^3p^2)$, then $\TV\left(G(n,p),G(n,p,d)\right)$ tends to zero as $n$ converges to infinity. Consequently, in this regime of $d$, the random graph $G(n,p,d)$ synchronizes as long as the $G(n,p)$ does.

However, it is unclear whether the random graph $G(n,p,d)$ synchronizes with high probability outside this regime. Our main result shows that it is indeed possible. 

\begin{theorem}(Main Result - First Part) \label{MainThm_1}
Let $G$ be a random graph sampled from $G(n,p,d)$. For any constant $\alpha > 0$, there are absolute constants $C_p,C_d>0$ such that if $d\ge C_d(n^2p^2+\log^4(n))\log^4(n)$, $\frac{\alpha}{n} < p < \frac{1}{2}$ and $np\ge C_p \log(n)^2$, then, with probability at least $1-n^{-\Omega(1)}$, the graph $G$ synchronizes.
\end{theorem}

Our proof strategy is to verify that the three conditions from Theorem \ref{thm:kassabovetal} are satisfied for $G(n,p,d)$ in the regime of Theorem \ref{MainThm_1}. We start by collecting a few preliminary results about random graphs. The first one is a standard fact, but we could not locate the exact statement in the literature, so we include here a short proof.
\begin{lemma}
\label{lemma:G(n,p)-A}
Let $A$ be the adjacency matrix of a graph sampled from the  $G(n,p)$. Suppose that $np>\log(n)$. Then, there exist absolute constants $K,c>0$ such that, with probability at least $1-e^{-t^2/c}$,
\begin{equation*}
\| A - \mathbb{E}A \| \le K\sqrt{np} + t.
\end{equation*}
\end{lemma}

\begin{proof}
By Theorem 2.1 in \cite{seginer2000expected}, we know that $\mathbb{E}\|A-\mathbb{E}A\|$ is at most a constant times the expectation of the maximum between the Euclidean norms of the columns. Based on this fact, an easy computation gives that $\mathbb{E}\|A-\mathbb{E}A\| \le C\sqrt{np}$. To obtain the tail bound, we apply Talagrand convex Lipschitz concentration inequality \cite[Theorem 5.2.16]{vershynin2018high}.
\end{proof}
\noindent The second preliminary proposition that we need is the following.

\begin{proposition}\cite[Proposition 1.3]{TestHighDim21}
\label{Prop1.3}
For any constant $\alpha> 0 $, there exist constants $C_d,C_\varepsilon >0$ such that if $\frac{\alpha}{n}\leq p\leq \frac{1}{2}$ and $d\geq C_d (n^2p^2+\log^4n)\log^4n$, then, for any $\varepsilon \geq C_\varepsilon \sqrt{\frac{1}{d}(np+\log n)\log^4n}$, there exists a probability space such that  $G_{-} \sim G(n,p(1-\varepsilon))$, $G\sim G(n,p,d)$ and $G_{+}\sim G(n,p(1+\varepsilon))$ satisfy, with probability at least $1-n^{-\Omega(\log n)}$, $G_{-} \subseteq G \subseteq G_{+}$.
\end{proposition}
For simplicity, we may assume that both random graphs have self-loops, i.e, every node may be connected with itself. We remark that this assumption does not change the energy function of the Kuramoto model. Next, we isolate the main technical proof in the following lemma that verifies the first condition of Theorem \ref{thm:kassabovetal}.

\begin{lemma}
\label{lemma:First_Condition}
Let $A$ be the adjacency matrix of a graph sampled from $G(n,p,d)$. There exists absolute constants $K,C_{\varepsilon},C_d>0$ such that if $np>\log(n)^2$ and $d\ge C_d(n^2p^2+\log^4(n))\log^4(n)$, then, with probability at least $1-n^{-\Omega(1)}$,
\begin{equation*}
\frac{\| \DE_A \|}{np} < \sqrt{\frac{2}{np}}\left(K + \max\left(4\frac{C_{\varepsilon}}{\sqrt{C_d}},24\right)\right).
\end{equation*}
\end{lemma}
\begin{proof}
We choose $C_{\varepsilon},C_d>0$ to be the constants in Proposition \ref{Prop1.3} and $K>0$ as in Lemma \ref{lemma:G(n,p)-A}. We fix
\begin{equation}
\label{eqEpsilon}
  \varepsilon:= \max\left(\frac{C_{\varepsilon}}{\sqrt{C_d}},6\right)\sqrt{\frac{2}{np}} \ge C_\varepsilon \sqrt{\frac{1}{d}(np+\log n)\log^4n} .
\end{equation}
The inequality above follows easily from the choice of $d$ and $np$. The necessity of the constant $6$ will become clear in what follows. Notice that the hypothesis of Proposition \ref{Prop1.3} is satisfied, so if $G^- \coloneqq G(n,p(1-\varepsilon))$ and $G^+ \coloneqq G(n,p(1+\varepsilon))$ are Erd\H{o}s-R\'enyi random graphs, then there exists a probability space such that $G^- \subseteq G \subseteq G^+$ with the required probability. Now we define $L \coloneqq D - A$, the Laplacian matrix associated to $G$, and $L^- \coloneqq D^- - A^-$, $L^+ \coloneqq D^+ - A^+$ the Laplacian matrices of $G^-$ and $G^+$ respectively.
Recall that the quadratic form of a Laplacian matrix is given by $ \frac{1}{2}\sum_{\{ i,j\} \in E(G)} (v_i - v_j)^2$. Hence the fact that $G^- \subseteq G \subseteq G^+$ implies the following monotone relation: For all vectors $v \in \mathbb{R}^n$, we have that
\begin{equation}\label{Sandwitch}
        v^T L^- v \leq v^T L v \leq v^T L^+ v,
    \end{equation}
therefore, 
    \begin{equation*}
        \begin{aligned}
            v^T(A - \Ex A)v 
            &= v^T A v - pv^T J_n v 
            = v^T D v - v^T L v - pv^T J_n v \\
            &\leq v^T D^+ v - v^T L^- v - pv^T J_n v + v^T \Ex[A^-] v -v^T \Ex[A^-] v \\
            &= v^T(D^+ - D^-) v + v^T(A^- - \Ex[A^-]) v - p\varepsilon v^T J_n v \\
            &\leq v^T(D^+ - D^-) v + v^T(A^- - \Ex[A^-]) v,
        \end{aligned}
    \end{equation*}
where we used here the self-loop assumption, $\Ex A = p J_n$ and $\Ex[A^-] = p(1-\varepsilon) J_n$. In the second line we use property \eqref{Sandwitch} and the fact that $G \subseteq G^+$ implies that $D_{ii} \leq D_{ii}^+$, for every $i\in [n]$. The last line then follows because $v^T J_n v \geq 0$ for every $v \in \mathbb{R}^n$.
    In an analogous way, we obtain that
    \begin{equation*}
        v^T(A - \Ex A)v 
        \geq v^T(D^- - D^+) v + v^T(A^+ - \Ex[A^+]) v.
    \end{equation*}
Putting everything together, we obtain
    \begin{equation} \label{bound}
        \begin{aligned}
            \| A -\Ex A \| 
            &= \sup_{\|v \|_2 = 1} |v^T (A-\Ex A) v | \\
            &\leq \|D^+ - D^- \| + \max \left\{\|A^- -\Ex[A^-]\|, \|A^+ -\Ex[A^+]\|\right\}.
        \end{aligned}
    \end{equation}
By Lemma \ref{lemma:G(n,p)-A}, we obtain that
    \begin{equation*}
        \begin{aligned}
            \|A^--\Ex[A^-]\| &\leq K\sqrt{np(1-\varepsilon)}, \\
            \|A^+-\Ex[A^+]\| &\leq K\sqrt{np(1+\varepsilon)},
        \end{aligned}
    \end{equation*}
    with the desired probability, where $K>0$ is the constant in Lemma \ref{lemma:G(n,p)-A}. For the first term in the right hand side of \eqref{bound}, we claim that $\|D^+ - D^- \| \leq 4p\varepsilon n$ with the desired probability. Indeed, each entry in the diagonal matrix $D^+-D^-$ is the sum of $n$ independent Bernoulli random variables:
    \begin{equation*}
        (D^+-D^-)_{ii} = \sum_{j= 1}^n X_{ij} \quad \text{for}\quad X_{ij} =
        \begin{cases}
            1 \ \text{if} \ A^+_{ij} = 1 \ \text{and} \ A^-_{ij} = 0,\\
            0 \ \text{else},
        \end{cases}
    \end{equation*}
    and
    \begin{equation*}
        \p(X_{ij} = 1) = 1-(\p(A^+_{ij} = 0) + \p(A^-_{ij} = 1)) = 1 - (1-p(1+\varepsilon) + p(1-\varepsilon)) = 2p\varepsilon.
    \end{equation*}
    The expectation of each entry of $D^+ - D^-$ is  $\Ex[(D^+-D^-)_{ii}] = 2p\varepsilon n$. We now apply Chernoff's deviation inequality \cite{vershynin2018high,janson2011random} to get
    \begin{equation*}
        \p((D^+-D^-)_{ii} \geq 4p\varepsilon n) \leq \exp\left(- \frac{p}{3}\varepsilon n\right).
    \end{equation*}
    By union bound, definition of $\varepsilon$ in \eqref{eqEpsilon} and the fact that $np > \log^2(n)$, we obtain that the claim holds with probability at least
    \begin{equation*}
        1-\exp\left(-\frac{p\varepsilon n}{3}+\log n\right) \ge 1-\exp\left(-2\sqrt{2} \log n+\log n\right),
    \end{equation*}
and then the above probability converges to one even if $n=2$. Now, for the range of $d$ in the hypothesis of this lemma, the following holds with the desired probability,
    \begin{equation*} 
    \begin{split}
        \frac{\| \DE_A\|}{np} &\leq K \frac{\max(\sqrt{1-\varepsilon},\sqrt{1+\varepsilon})}{\sqrt{np}} + \frac{4p\varepsilon n}{np} \\
        &\le K \sqrt{\frac{2}{np}} + 4\varepsilon.
     \end{split}
    \end{equation*}
Recall the choice of $\varepsilon$ in the beginning to conclude the proof.
\end{proof}
\noindent We now present the proof of the main theorem.
\begin{proof}[Proof of Theorem \ref{MainThm}]
As explained above, we verify that the three conditions listed in Theorem \ref{thm:kassabovetal} are satisfied. Thanks to Lemma \ref{lemma:First_Condition}, we know that $\frac{\| \DE_A\|}{np} < \frac{1}{12}$ with the required probability. To verify the second condition, we apply triangular inequality to obtain
\begin{equation*}
\| \DE_L \| = \| D- A -\Ex D + \Ex A \| \leq \|D-\Ex D \| + \| A - \Ex A \| < \|D-\Ex D \| + \frac{1}{12}.
\end{equation*}
To deal with the first term in the right hand side, observe that the entries of $D-\Ex D $ are only nonzero on the diagonal and $ (D-\Ex D)_{ii} = \sum_{j=1}^n \mathds{1}_{\langle X_i,X_j\rangle \ge t_{p,d}}-p$. Conditioned on $X_i$, the distribution of $(D-\Ex D)_{ii}$ is a centered binomial random variable and we apply Chernoff's deviation again to obtain, for every $\delta >0$,
 \begin{equation*}
    \p \left( \left|(D-\Ex D)_{ii}\right| \geq \delta np\right)
        \leq 2 \exp \left(-\frac{\delta^2 np}{3}\right).
    \end{equation*}
We choose $\delta = \frac{\log n}{\sqrt{np}}$ and union bound over all $i\in [n]$ to obtain that, with the required probability,
\begin{equation*}
\frac{\| \DE_L \|}{np} < \frac{\log n \sqrt{np}}{np} + \frac{1}{12} \le \frac{1}{\sqrt{C_p}} + \frac{1}{12} < \frac{1}{8},
\end{equation*}
for a sufficiently large $C_p>0$. The reason for $\frac{1}{8}$ instead of $\frac{1}{4}$ will become clear in what follows. Next, we proceed to verify the last condition. It is easy to see that for a sufficiently large $C_p>0$, we have
\begin{equation*}
\frac{12\|\DE_A\|}{np} \le \frac{1}{4\log n}.
\end{equation*}
Since the function $\sin^{-1}(x)$ is non-increasing and satisfies $\sin^{-1}(x) < 2x$ for $x \geq 0$, we obtain that
\begin{equation}
\label{LHSBound}
        \sin^{-1}\left(\frac{12\|\DE_A\|}{np}\right) \leq \frac{1}{2\log n}.
\end{equation}
On the other hand, we know that $\frac{np}{2\|\DE_L\|} > 4$ and then 
\begin{equation}
\label{RHSBound}
\log\left(\frac{np}{2\|\DE_L\|} -1\right) > \log(3).
\end{equation}
Combining \eqref{LHSBound} with \eqref{RHSBound}, we get 
\begin{equation*}
\frac{\pi/4}{\sin^{-1}\left(\frac{12\|\DE_A\|}{np} \right)} \ge \frac{\pi}{2}\log n > \log\left(\frac{n}{6}-3\right) + 1 >  \frac{\log(n/6)}{\log (\frac{np}{2\|\DE_L\|}-1)} + 1.
\end{equation*}
We conclude the proof by applying Theorem \ref{thm:kassabovetal}.
\end{proof}

\section{Improvement when $p$ is sublinear.}
\label{sec:sublinear}
The downside of Theorem \ref{MainThm_1} is that it requires larger values of $d$ as $p$ increases. To fix this issue, we use another result in the literature that provides sharper bounds when $np=\Omega(\log(n)^8)$. The main result of this section covers the second regime in Theorem \ref{MainThm}.
\begin{theorem}(Main Result - Second Part)
Let $G$ be a random graph sampled from $G(n,p,d)$. There exists an absolute constant $C_d>0$ and $C_n>0$ such that if $d=C_d\log^{3}(n)$ and $np=C_n\log^{10}(n)$, then with high probability, the graph $G$ synchronizes. 
\end{theorem}
\noindent We highlight that the same result holds for any range of $p$ as long as $p=o(1)$. The startpoint of our analysis is the following result.
\begin{proposition}[Theorem 1.7 \cite{liu2022local}]
Let $G(n,p,d)$ be a random geometric graph such that $\mathbb{P}(\langle X_i,X_j\rangle \ge \tau(p,d))=p$. Then with high probability,
\begin{equation*}
    \frac{\|\DE A\|}{np} \le \max\left\{(1+\varepsilon(d,\tau))\tau,\frac{\log^4(n)}{\sqrt{np}}\right\}.
\end{equation*}
Here $\varepsilon(d,\tau)$ is a function that goes to zero as $d\tau(p,d)^2$ goes to infinity.
\end{proposition}
The main application of the previous proposition is to prove the following corollary.
\begin{corollary}
Let $G(n,p,d)$ be a random geometric graph such that $np\ge C_n\log^{10}(n)$ with $C_n>2304$. There exists an absolute constant $C_d>0$ such that, for $d=C_d\log^3(n)$, with high probability,
\begin{equation*}
    \frac{\|\DE A\|}{np} < \frac{1}{48\log n}<\frac{1}{12}.
\end{equation*}
\end{corollary}
Once the corollary is proved, the proof of the main result follows exactly the same steps as before and therefore it will be omitted. The following lemma collects a few key facts about random geometric graphs, the proof can be found in \cite{brieden2001deterministic}. 
\begin{lemma}[Equations (3.5) and (3.6) in \cite{lugosi2017lectures}]
\label{distribution_lemma}
Let $G(n,p,d)$ be a random geometric graph. Assume that $p <1/2$ is fixed, then
\begin{equation*}
    \lim_{d\rightarrow \infty} \tau \sqrt{d} = \Phi^{-1}(1-p).
\end{equation*}
Here $\Phi(x)$ is the cumulative distribution function of a standard Gaussian random variable. In particular, for $p$ small enough, $\tau\sqrt{d}\ge 2$. Moreover, for this regime,
\begin{equation*}
    \frac{1}{6\sqrt{d}\tau}(1-\tau^2)^{(d-1)/2} \le p \le \frac{1}{2\sqrt{d}\tau}(1-\tau^2)^{(d-1)/2}.
\end{equation*}
\end{lemma}

\begin{proof}
By assumption, $\log^4(n)/\sqrt{np}$ is smaller than $1/(48\log n)$. We claim that, for the choice of $p$ and $d$, we have $\tau <1/(48\log n)$ and $d\tau^2$ goes to infinity. Suppose by contradiction that $\tau$ is larger than $1/(48\log n)$, then 
\begin{equation*}
   p\le \frac{1}{2\sqrt{d}\tau}e^{-\tau^2(d-1)/2} \le \frac{24}{\sqrt{C_d\log n}}e^{-(C_d-1)\log n / 4608} \le n^{-(C_d-1)/4608} .
\end{equation*}
Notice that $p \ge C_n\log^{10}(n)/n = n^{-1+o(1)}$. This gives the contradiction if we choose a sufficiently large constant $C_d>0$. To get the second conclusion, we assume again by contradiction that $d\tau^2$ converges to some limit $L <\infty$. By lemma \ref{distribution_lemma}, if $p$ is a constant, then $\tau$ is of order $1/\sqrt{d}$. Since $p$ vanishes and $p$ is inversely proportional to $\tau$, we can assume that for $n$ large enough, it holds that $\tau \ge 2/\sqrt{d}$. Therefore, it holds that 
\begin{equation*}
    \frac{1}{6\sqrt{d\tau^2}}\left(1-\frac{\tau^2d}{d}\right)^{(d-1)/2}\le p.
\end{equation*}
It is easy to see that the left hand side has positive limit, but the right hand side converges to zero by assumption. We conclude that $d\tau^2$ goes to infinity and the claim is proved.
\end{proof}

\section*{Acknowledgment}
We would like to thank Alex Townsend and Tselil Schramm for insightful discussions. In particular, we are indebted to Tselil Schramm for various insights and clarifications with regards to properties of random geometric graphs.


\end{document}